\documentclass[preprint,authoryear,12pt]{elsarticle}

\usepackage{amssymb}
\usepackage{amsthm}

\newtheorem{theorem}{Theorem}

\newtheorem{lemma}{Lemma}

\newtheorem{definition}{Definition}
\newtheorem{remark}{Remark}

\usepackage{dsfont}
\usepackage{mathtools}
\usepackage{graphicx}

\journal{Statistics \& Probability Letters}

\begin{document}

\begin{frontmatter}

\title{Mean and dispersion of harmonic measure}

\author[a]{Sirio Legramanti\corref{mycorrespondingauthor}}
\cortext[mycorrespondingauthor]{Corresponding author}
\ead{sirio.legramanti@phd.unibocconi.it}
\address[a]{Department of Decision Sciences, Bocconi University, Via Roentgen 1, 20136 Milan, Italy}

\begin{abstract}
	In this note, we provide and prove exact formulas for the mean and the trace of the covariance matrix 
	of harmonic measure, regarded as a parametric probability distribution.
\end{abstract}

\end{frontmatter}


\section{Introduction} 
\label{sec_intro}

The term \textit{harmonic measure} was introduced by \cite{nevanlinna1934harmonische} in the context of partial differential equations. 
Ten years later, \cite{kakutani1944} provided a probabilistic interpretation linked to Brownian motion: the harmonic measure on a regular bounded domain in $ \mathds{R}^d $ is the exit distribution of a Brownian motion started in a point $ \theta $ within the domain (see Theorem \ref{thm_kakutani}). The starting point $ \theta $ serves as a parameter for such distribution.\\ 
A few decades have passed but, up to our best knowledge, there is no explicit and direct reference in the literature for some of the basic results that are usually available for parametric distributions, as for example formulas for the mean and for some measure of dispersion. When such functionals are of interest, exact formulas can save simulations that would consume time and resources and would only offer approximate results. Such formulas can also highlight the role of the distribution parameters, especially if they are easy to interpret as in our case. Moreover, they can be useful in statistical inference, for example to derive the properties of the estimators.\\
This note provides exact formulas for the mean of the harmonic measure on the boundary of any regular bounded domain in $\mathds{R}^d$ and for the trace of the covariance matrix of the harmonic measure on the boundary of balls in $\mathds{R}^d$.\\
The remainder of this paper is organized as follows. In Section \ref{sec_harmonic} we recall basic definitions and properties of harmonic measure. In Section \ref{sec_results} we provide the above-mentioned formulas. In Section \ref{sec_application} we illustrate an application of such formulas to the assessment of privacy for GPS trajectories sharing. Finally, in Section \ref{sec_discussion} we draw conclusions and discuss possible future research directions.

\section{Harmonic measure} 
\label{sec_harmonic}

A definition of the harmonic measure on the border of a domain in $ \mathds{R}^d $ is the following.
\begin{definition} \label{def_harmonic_measure_PDE}
	\textbf{(Harmonic measure)} Let $ D $ be a domain (open, connected) in $ \mathds{R}^d $. For any Borel measurable subset $E$ of its boundary $\partial D$, the boundary data $f = \mathds{1}_E$ has a harmonic extension into $D$, which is called the \textit{harmonic measure} of $E$ and is denoted by $\omega(x,E,D)$, with $x \in D$.
\end{definition}
As a function of $x$, $\omega(x,E,D)$ is harmonic on $ D $ for any Borel measurable subset $E$ of its boundary $\partial D$. On the other hand, as a function of $E$, $\omega(x,E,D)$ is a measure on  $\partial D$ for any fixed $ x \in D $ \citep[Theorem 3.10]{hayman}.\\
Consider now a Markov process $ \{X_t\}_{t \in \mathds{R}^+} $ taking values in $\mathds{R}^d$ and denote with $\tau_D$ its exit time from a domain $D$ (i.e. its hitting time of $ D^c $).
We say that $D$ is \textit{regular} if and only if all its boundary points are regular in the following sense.
\begin{definition}
	\textbf{(Regular boundary point)}. Let $ D $ be a domain in $ \mathds{R}^d $. We say that $z \in \partial D$ is a \textit{regular boundary point} of $D$ if and only if
	\begin{equation} \label{reg_boundary_pt}
		\mathds{P}\{\tau_D=0 \mid X_0 = z \}=1
	\end{equation}
\end{definition}
The following theorem, that first appeared in \citet{kakutani1944}, provides a probabilistic interpretation of harmonic measure on regular bounded domains.
\begin{theorem} \label{thm_kakutani}
	Let $ \{X_t\}_{t \in \mathds{R}^+} $ be a Brownian motion in $\mathds{R}^d$. Then, for any regular bounded domain $ D $ in $ \mathds{R}^d $, any Borel measurable subset $E$ of its boundary $\partial D$ and any $x \in D$:
	\begin{equation} \label{eq_kakutani}
	\omega(x,E,D) = \mathds{P} \{ X_{\tau_D} \in E \mid X_0 = x \}.
	\end{equation}
\end{theorem}
\begin{proof}
	Follows immediately from \citet[Theorems 1.23, 1.24]{chung}, plugging in $ f=\mathds{1}_E $ and recalling that 
	$ \mathds{P} \{ \tau_D < \infty \mid X_0 = x \} = 1 $ for any $x \in D$, since $ D $ is bounded
	\citep[Chapter 4.2, Corollary to Property (X)]{chung1982lectures}.
\end{proof}
If a bounded domain $D$ has Poisson kernel $K$ on $ D \times \partial D $, then - by definition - the harmonic measure on $\partial D$ can be expressed as:
\begin{equation} \label{eq_poisson_kernel}
\omega (x,E,D) = \int_E{K(x,y) \sigma(dy)}.
\end{equation}
where $\sigma$ is a Borel measure on $\partial D$ \citep[Chapter 1.4]{chung}. For $C^1$ domains, the 
the usual $(d-1)$ dimensional Lebesgue measure 
is taken as $ \sigma $.
When available, the Poisson kernel is then the Radon-Nikodim derivative of the harmonic measure with respect to $\sigma$. 
If - in the light of Theorem \ref{thm_kakutani} - we regard harmonic measure as a probability measure, its Poisson kernel can then be interpreted as its probability density function.\\
For example, if $D$ is a ball of center $ c $ and radius $ r $ - that we will denote with $ B(c,r) $ - the Poisson kernel is available, with the following explicit formula.
\begin{theorem} \label{thm_ball_kernel}
	For any ball $ B(c,r) $ in $ \mathds{R}^d $, 
	the Poisson kernel is
	\begin{equation} \label{eq_balls_kernel}
	K(x,y) = \frac{\Gamma(d/2)}{2\pi^{d/2}r} \; \frac{r^2-\|x-c\|^2} {\|x-y\|^d}
	\end{equation}
	where $ x \in B(c,r) $, $ y \in \partial B(c,r) $, $ \Gamma(\cdot) $ is the Gamma function and $ \| \cdot \| $ is the Euclidean norm in $\mathds{R}^d$.
\end{theorem}
\begin{proof}
	See \citet[Theorem 1.13]{chung}.
\end{proof}

\section{Main results} 
\label{sec_results}

In this section, we adopt the probabilistic interpretation of harmonic measure and hence refer to it as a probability distribution. In particular, we say that a random vector $ Y $ on the boundary $ \partial D $ of a domain $ D $ in $ \mathds{R}^d $ is distributed according to the harmonic measure on $ \partial D $ parametrized by $ \theta \in D $, and we write $ Y \sim \mathds{H}_\theta^D $, if for any Borel measurable subset $E$ of $\partial D$:
\begin{equation}
\mathds{P} \{ Y \in E \} = \omega(\theta,E,D).
\end{equation}
By Theorem \ref{thm_kakutani}, if $ D $ is bounded and regular, this is equivalent to $ Y $ having the same distribution as the exit point from $ D $ of a Brownian motion started in $ \theta $.\\

The following lemma provides the formula for the mean of the harmonic measure on the boundary of any regular bounded domain in $ \mathds{R}^d $.
\begin{lemma} \label{lemma_mean}
	Let $ D $ be a regular bounded domain in $ \mathds{R}^d $. 
	The mean of the harmonic measure on $ \partial D $ parametrized by $ \theta \in D $ is 
	$\theta$ itself. In symbols:
	\begin{equation*}
	Y \sim \mathcal{H}_\theta^D \
	\implies \ \mathds{E}[Y] = \theta.
	\end{equation*}
\end{lemma} 
\begin{proof}
Consider the sequence of random vectors $ \{Y_n\}_{n \geq 0} $ defined as follows:
	\begin{equation} 	\label{martingale}
	Y_0 = \theta, \quad \quad
	Y_{n+1} \ | \ Y_n \ \sim \ \mathcal{H}_{Y_n}^{B(Y_n,r_n)}, \quad \quad \mbox{with} \quad r_n = \frac{1}{2} \ dist(Y_n,\partial D).
	\end{equation}
	This construction is adapted from the one in \citet[Appendix F]{garnett2005harmonic}. 
	By Theorem \ref{thm_kakutani}, this is equivalent to $ Y_{n+1} $ being the exit point from $ B(Y_n,r_n) $ of a Brownian motion started in $ Y_n $, with $ r_n $ as in \eqref{martingale}. 
	By symmetry, $ Y_{n+1} $ is uniformly distributed over $ \partial B(Y_n,r_n) $ and $\mathds{E}[Y_{n+1}|Y_n]=Y_n$. Hence, $ \{Y_n\}_{n \geq 0} $ is a martingale with respect to the natural filtration and $ \mathds{E}[Y_n] = 
	\theta $ for all $ n \geq 0 $.
	Since $ Y_n \in D $ for all $ n \geq 0 $ and $ D $ is bounded, such martingale is bounded in $ L^1 $ and hence, by Doob's Forward Convergence Theorem \citep[Theorem 11.5]{williams1991probability}, it converges almost surely to a random vector $Y$. 
	As a consequence, almost surely (as everything that follows), the sequence $ \{ Y_n \}_{n \geq 0} $ is Cauchy and hence $ dist(Y_n,Y_{n+1}) \to 0 $. But $ dist(Y_n,Y_{n+1}) = r_n =  dist(Y_n,\partial D) / 2 $, 
	which implies that $dist(Y_n,\partial D) \to 0$ and hence $Y \in \partial D$. Specifically, $ Y $ is the exit point from $D$ of a Brownian motion started in $ \theta $, whose trajectory up to its first exit from $ D $ can be reconstructed joining all the stretches of Brownian motion leading from $ Y_{n} $ to $ Y_{n+1} $. Hence, by Theorem \ref{thm_kakutani}, $Y \sim \mathcal{H}_\theta^D$.
	Since D is bounded, applying Dominated Convergence Theorem \citep[Theorem 5.9]{williams1991probability} componentwise, we can conclude that:
	\begin{equation*}
	\mathds{E}[Y] = \mathds{E}[\lim_{n \to \infty} Y_n] = \lim_{n \to \infty} \mathds{E}[Y_n] = \lim_{n \to \infty} \theta = \theta.
	\end{equation*}
\end{proof}
\begin{remark}\label{rem_mean}
Lemma \ref{lemma_mean} holds for the exit distribution from a regular bounded domain of any symmetric stochastic process in $ \mathds{R}^d $ with the strong Markov property.
\end{remark}
The following lemma provides the formula for the trace of the covariance matrix of the harmonic measure on the boundary of a ball in $ \mathds{R}^d $.
\begin{lemma} \label{lemma_trace}
	The trace of the covariance matrix of the harmonic measure on the boundary of a ball $ B(c,r) $ in $ \mathds{R}^d $ 
	parametrized by $\theta \in B(c,r)$ is equal to $ r^2- \| \theta - c \|^2 $. In symbols:
	\begin{equation}\label{eq_trace_lemma}
	Y \sim \mathcal{H}_\theta^{B(c,r)} 
	\ \implies \ tr(\Sigma_Y) = r^2-\| \theta - c \|^2.
	\end{equation}
\end{lemma}
\begin{proof}
	First assume $ c=0 $.
	Since $ B(0,r) $ is bounded and regular, by Theorem \ref{thm_kakutani}, $ Y \stackrel{d}{=} X_\tau $, where $ \{X_t\}_{t \in \mathds{R}^+} $ is a Brownian motion starting in $ \theta $ and $ \tau $ is its first exit point from $ B(0,r) $. By the law of total variance, for each coordinate $ Y_i $ of $ Y $ ($ i=1,\ldots,d $):
	\begin{equation*}
	Var(Y_i) = \mathds{E}[Var(Y_i | \tau)] + Var(\mathds{E}[Y_i|\tau]).
	\end{equation*}
	Recalling that $ Y_i $ is the position of a univariate Brownian motion at time $ \tau $ and exploiting the formula for the expected exit time of a Brownian motion from $ B(0,r) $ provided in \citet[Example 7.4.2]{oksendal2003stochastic}, we get
	\begin{equation*}
	\mathds{E}\left[Var(Y_i | \tau)\right] = \mathds{E}[\tau] = \frac{1}{d} \left(r^2-\|\theta\|^2 \right)
	\end{equation*}
	while, by Lemma \ref{lemma_mean},
	\begin{equation*}
	Var(\mathds{E}[Y_i|\tau]) \leq Var(\mathds{E}[Y_i]) = Var (\theta_i) = 0.
	\end{equation*}
	Putting all previous equations together, we have
	\begin{equation*}
	tr(\Sigma_Y) = \sum_{i=1}^d Var(Y_i) = \sum_{i=1}^d \frac{1}{d} \left( r^2-\|\theta\|^2 \right) = r^2-\|\theta\|^2.
	\end{equation*}
	The formula for general $ c \in \mathds{R}^d $ follows by translation.
\end{proof}
\begin{remark} \label{rem_trace}
	Let $ Y $ be a random vector in $ \mathds{R}^d $ with $ Y \sim \mathcal{H}_\theta^{B(c,r)} $.
	For $ d=1 $, $ tr(\Sigma_Y) = Var(Y) = \mathds{E} \left[ | Y - \theta |^2 \right]$. Similarly, for higher dimensions:
	\begin{equation*}
	tr(\Sigma_Y) = \sum_{i=1}^d Var(Y_i) = \mathds{E}\left[ \sum_{i=1}^d | Y_i - \theta_i |^2 \right] = \mathds{E} \left[ \| Y - \mathds{E}[Y] \|^2 \right] = \mathds{E} \left[ \| Y - \theta \|^2 \right].
	\end{equation*}
\end{remark}
\begin{remark}\label{rem_2dim}
For $ d=2 $, equation \eqref{eq_trace_lemma} is easily obtained by computing directly $ \mathds{E} \left[ \| Y - \theta \|^2 \right] $, exploiting Remark \ref{rem_trace} and the probability density function of $ Y $ provided in Theorem \ref{thm_ball_kernel}:
\begin{equation*}
tr(\Sigma_Y) = \mathds{E} \left[ \| Y - \theta \|^2 \right] = \frac{1}{2\pi r} \ \int_{\|y\|=r}  \|y-\theta\|^2 \  \frac{r^2-\|\theta-c\|^2}{\|y-\theta\|^2} \ dy =  r^2-\|\theta-c\|^2.
\end{equation*}
\end{remark}
In the light of Theorem \ref{thm_kakutani}, the interpretation of Lemma \ref{lemma_trace} is straightforward: the closer the starting point of the Brownian motion is to the boundary of the ball, the more concentrated the exit distribution. The most spread exit distribution is obtained when the starting point is the center of the ball, producing an exit distribution which is uniform over the boundary.\\

As a double-check and example, in Table \ref{tab_sims} we compare the theoretical and empirical mean and trace of the covariance matrix of the exit point from $ B(0,1) $ of a $ d $-dimensional Brownian motion started in $ \theta $, for different dimensions $ d $ and  starting points $ \theta $. Theoretical values are based on Lemma \ref{lemma_mean} and Lemma \ref{lemma_trace}. Empirical estimates are based on a sample of $ 500 $ exit points for each setting. Brownian motion was simulated at a timestep of $ 10^{-4} $.

\begin{table}
\resizebox{\textwidth}{!} {
\begin{tabular}{ccccccccccc}
		d & $ \theta_1 $ & $ \theta_2 $ & $ \theta_3 $ & $ \theta_4 $ & $ \bar{y}_1 $ & $ \bar{y}_2 $ & $ \bar{y}_3 $ & $ \bar{y}_4 $ & $ r^2-\|\theta-c\|^2 $ & $ tr(\widehat{\Sigma_Y}) $\\
	\hline
2 & 0.20 & 0.00 &  &  & 0.17 & 0.03 &  &  & 0.96 & 0.98 \\ 
2 & 0.50 & 0.00 &  &  & 0.45 & -0.01 &  &  & 0.75 & 0.81 \\ 
2 & 0.80 & 0.00 &  &  & 0.82 & -0.02 &  &  & 0.36 & 0.35 \\ 
3 & 0.20 & 0.00 & 0.00 &  & 0.23 & -0.01 & 0.03 &  & 0.96 & 0.96 \\ 
3 & 0.50 & 0.00 & 0.00 &  & 0.49 & -0.00 & -0.00 &  & 0.75 & 0.77 \\ 
3 & 0.80 & 0.00 & 0.00 &  & 0.79 & -0.01 & 0.01 &  & 0.36 & 0.39 \\ 
4 & 0.20 & 0.00 & 0.00 & 0.00 & 0.20 & -0.03 & 0.02 & 0.03 & 0.96 & 0.97 \\ 
4 & 0.50 & 0.00 & 0.00 & 0.00 & 0.49 & -0.05 & 0.00 & -0.02 & 0.75 & 0.77 \\ 
4 & 0.80 & 0.00 & 0.00 & 0.00 & 0.80 & -0.02 & -0.02 & 0.01 & 0.36 & 0.37 \\ 
	\hline
\end{tabular}
}
\caption{Theoretical and empirical mean and trace of the covariance matrix of $ Y $, the exit point from $ B(0,1) $ of a $ d $-dimensional Brownian motion started in $ \theta $, for different dimensions $ d $ and  starting points $ \theta $.}
\label{tab_sims}
\end{table}

\section{An application to privacy for GPS trajectories sharing}
\label{sec_application}

The results in Section \ref{sec_results} are very useful, for example, in addressing the privacy issues related to GPS trajectories, that nowadays are massively recorded and shared due to 
the diffusion of GPS sensors and the success of fitness apps.
Unfortunately, GPS trajectories can be exploited to locate a user's house \citep{liao2006location,hoh2006enhancing,krumm2007inference}, endangering both people and property.\\
Possible countermeasures rely on perturbing or cutting original trajectories as much as necessary to meet the required privacy standards.
A statistically sound way to assess the amount of privacy guaranteed by a countermeasure is to regard a privacy attack as a parameter estimation problem, with the parameter of interest being the user's house location, and measure the efficacy of the countermeasure by the quality of the estimation based on the perturbed/cut trajectories: the better the estimation, the poorer the privacy.\\
A family of possible countermeasures is \textit{spatial cloaking}, which consists in hiding the part of the trajectory within a given privacy area \citep{gruteser2003anonymous}.
Without loss of generality, let's assume that such area is a domain in $ \mathds{R}^2 $.  
If we assume that the user's motion is memoryless and always starts in the user's house, trajectories exit points from the privacy area are sufficient statistics for the user's house location.
In such case, properties of the exit distribution - in particular its mean and dispersion - become fundamental to assess the quality of the estimation and hence the level of privacy.
If users' motion is symmetric, Remark \ref{rem_mean} implies that the sample mean of the exit points is an unbiased estimator of the user's house location. Results about the dispersion of the exit distribution are then necessary to quantify the dispersion of such estimator. 
Lemma \ref{lemma_trace} answers to this question for the exit distribution of a Brownian motion from a ball, but further research is needed for differently shaped domains or different stochastic processes.
Still, even if Brownian motion is not the most realistic model for the movement of vehicles in a city, it can be useful in an exploratory analysis, in particular to rule out candidate countermeasures: if a countermeasure cannot hide the start of a Brownian motion, it cannot aim to hide the start of the much more structured human movement.

\section{Discussion}
\label{sec_discussion}

In the present note, we have approached harmonic measure as a parametric probability distribution and we have provided formulas for its mean and for a measure of its dispersion.
In particular, we have proved exact formulas for the mean of the harmonic measure on the boundary of any regular bounded domain in $\mathds{R}^d$ and for the trace of the covariance matrix of the harmonic measure on the boundary of balls in $\mathds{R}^d$.
Such formulas can save simulations, offer better insights on the role of the distribution parameters and are useful in statistical inference.
Further research might focus on formulas for the trace of the covariance matrix of the harmonic measure on a broader class of domains or for 
stochastic processes other than Brownian motion, as suggested by the application in Section \ref{sec_application}.

\section*{Acknowledgments}

The author is grateful to Giacomo Aletti and Daniele Durante for their helpful comments on a first version of this work.

\bibliographystyle{model2-names}
\bibliography{harmonic}


\end{document}